\documentclass[10pt,a4paper]{amsart}
\usepackage{amsthm, amsfonts, amssymb,latexsym}
\usepackage{enumerate, color}
\usepackage[english]{babel}
\usepackage[latin1]{inputenc}

\newtheorem{prop}{Proposition}
\newtheorem{lemma}{Lemma}
\newtheorem{theorem}{Theorem}
\newtheorem{cor}{Corollary}
\newtheorem{conj}{Conjecture}
\numberwithin{cor}{section} \numberwithin{theorem}{section}
\numberwithin{lemma}{section} \numberwithin{prop}{section}

\DeclareMathOperator{\lcm}{lcm}

\newcommand{\bnk}{\binom {[n]}{k}}

\newcommand{\A}{A}

\newcommand{\Z}{\mathbb{Z}}
\newcommand{\EE}{\mathbb{E}}
\newcommand{\PP}{\mathbb{P}}

\title{The least common multiple of random sets of positive integers}

\author[J. Cilleruelo]{Javier Cilleruelo}
\address{J. Cilleruelo: Instituto de Ciencias Matem\'aticas
(CSIC-UAM-UC3M-UCM) and Departamento de Matem\'aticas, Universidad
Aut\'onoma de Madrid, 28049 Madrid, Spain }
\email{franciscojavier.cilleruelo@uam.es}
\author[J. Ru\'e]{Juanjo Ru\'e}
\address{J. Ru\'e: Institut f\"ur Mathematik, Freie Universit\"at Berlin, Arnimallee 3-5, D-14195 Berlin, Germany} \email{ jrue@zedat.fu-berlin.de}
\author[P. \v Sarka]{Paulius \v Sarka}
\address{P. \v Sarka: Institute of Mathematics and Informatics,
Akademijos
4, Vilnius LT-08663, Lithuania and Department of Mathematics and
Informatics, Vilnius
University, Naugarduko 24, Vilnius LT-03225, Lithuania}
\email{paulius.sarka@gmail.com}
\author[A. Zumalac\'{a}rregui]{Ana Zumalac\'arregui}
\address{A. Zumalac\'{a}rregui: Instituto de Ciencias Matem\'aticas
(CSIC-UAM-UC3M-UCM) and Departamento de Matem\'aticas, Universidad
Aut\'onoma de Madrid, 28049 Madrid, Spain }
\email{ana.zumalacarregui@uam.es}

\begin{document}

\begin{abstract}
We study the typical behavior
of  the least common multiple of the elements of a random subset
$\A\subset \{1,\dots , n\}$. For example we prove that $\lcm\{a:\ a\in A\}=2^{n(1+o(1))}$ for almost all subsets $A\subset\{1,\dots ,n\}$.
\end{abstract}
\maketitle
\section{Introduction}\label{sec:intro}
The function $\psi(n)=\log \lcm \left \{m:\ 1\le m\le n \right\}$
was introduced by Chebyshev in his study on the distribution of the
prime numbers. It is a well known fact that the asymptotic relation
$\psi(n)\sim n$ is equivalent to the Prime Number Theorem,  which was proved independently by J. Hadamard and C.J. de la Vall\'ee Poussin.

In the present paper, instead of considering the whole set $\{1,\dots,n\}$, we study the typical behavior of the quantity $\psi(\A):=\log \lcm \{a:\ a\in \A\}$ for a random set $\A$ in $\{1,\dots ,n\}$ when $n\to \infty$. We define $\psi(\emptyset)=0$. We consider two natural models.

In the first one, denoted by $B(n,\delta)$,
 each element in $\A$ is chosen independently at random in $\{1,\dots, n\}$ with probability  $\delta=\delta(n)$, typically a function of $n$.
 \begin{theorem}\label{TH1}If $\delta=\delta(n)<1$ and $\delta n\to \infty$ then
$$\psi(\A)\sim n\frac{\delta\log(\delta^{-1})}{1-\delta}
$$
asymptotically almost surely in $ B(n,\delta)$ when $n\to \infty$ .
\end{theorem}
The case $\delta=1$ corresponds to  the classical Chebyshev function and its asymptotic estimate appears as the limiting case, as $\delta$ tends to 1, in Theorem \ref{TH1}, since $\lim_{\delta \to 1}\frac{\delta\log(\delta^{-1})}{1-\delta}=1$.

When  $\delta=1/2$ all the subsets $A\subset \{1,\dots,n\}$ are chosen with the same probability and Theorem \ref{TH1} gives the following result.
\begin{cor}
For almost all sets $\A\subset \{1,\dots,n\}$ we have that $$\lcm\{a:\ a\in A\}=2^{n(1+o(1))}.$$
\end{cor}
For a given positive integer $k=k(n)$, again typically a function of $n$, we consider the second model, where each subset of $k$ elements is chosen uniformly at random among  all sets of  size $k$ in $ \{1,\dots, n\}$. We denote this model by $S(n,k)$.

When $\delta=k/n$ the heuristic suggests that both models are quite similar. Indeed, this is the strategy we follow to prove Theorem \ref{TH2}.
\begin{theorem}\label{TH2}For $k=k(n)< n$ and $k\to\infty$ we have
$$\psi(\A)= k\frac{\log(n/k)}{1-k/n}\left (1+O(e^{-C\sqrt{\log k}})\right )
$$ almost surely in $S(n,k)$ when $n\to \infty$  for some positive constant $C$.
\end{theorem}
The case $k=n$, which corresponds to Chebyshev's function, is also obtained as a limiting case in Theorem \ref{TH2} in the sense that $\lim_{k/n\to 1}\frac{\log(n/k)}{1-k/n}=1$.

This work has been motivated by a result of the first author about the asymptotic behavior of $\psi(\A)$ when $\A=A_{q,n}:=\{q(m):\ 1\le q(m)\le n\}$ for a quadratic polynomial $q(x)\in \Z[x]$.  We wondered if that behavior was typical among the sets $\A\subset \{1,\dots ,n\}$ of similar size. We analyze this issue in the last section.
\section{Chebyshev's function for random sets in $B(n,\delta)$. Proof of Theorem
\ref{TH1}}\label{sec:expectation}
 The following lemma provides us
with an explicit expression for $\psi(\A)$ in terms of the Mangoldt function $$\Lambda(m)=\begin{cases}\log p\quad \text{ if } m=p^k \text{ for some }k\ge 1\\ 0,\qquad \text{ otherwise}.\end{cases}$$
\begin{lemma}\label{general}For any set of positive integers $\A$ we
have
$
 \psi(\A)= \sum_{m}\Lambda(m)I_{\A}(m),
$
where
$\Lambda$ denotes the classical Von Mangoldt function and
$$
I_{\A}(m)=\left\{\begin{array}{lc} 1 \qquad \text{ if
}\A\cap\{m,2m,3m,\dots\}\neq \emptyset, \\ 0 \qquad \text{ otherwise.}
\end{array}\right.
$$
\end{lemma}
\begin{proof} We observe that for any positive integer $l$, the number $\log l$ can be written as $\log l=\sum_{p^k\mid l}\log p$, where the sum is taken over all the powers of primes. Thus, using that $p^k|\lcm  \{a: a\in \A\}$ if and only if $\A\cap\{p^k, 2p^k,3p^k,\dots\}\neq \emptyset $, we get
$$\log \lcm (a:\ a\in A)=\sum_{p^k\mid \lcm (a:\ a\in A)}\log p=\sum_{p^k}(\log p) I_A(p^k)=\sum_m\Lambda(m)I_A(m).$$
\end{proof}
Note that if $\A=\{1,\dots,n\}$  then $\psi(\A)=\sum_{m\le
n}\Lambda(m)$ is the classical Chebychev function $\psi(n)$.
\subsection{Expectation}
First of all we give an explicit expression for the expected value of
the random variable $X=\psi(\A)$ where $\A$ is a random set in $B(n,\delta)$.
\begin{prop}\label{expectation} For the random variable $X=\psi(\A)$ in $B(n,\delta)$ we have
$$\EE\left(X\right)=n\frac
{\delta\log(\delta^{-1})}{1-\delta}+\delta\sum_{r\geq 1}
R\left(\frac{n}{r}\right) (1-\delta)^{r-1},$$
where $R(x)=\psi(x)-x$ denotes the error term in
the Prime Number Theorem.
\end{prop}
\begin{proof} The ambiguous case $\delta=1$ must be understood as the limit as $\delta\to 1$, which recovers the equality $\psi(n)=n+R(n)$. In the following we assume that $\delta<1$.
By linearity of the expectation, Lemma \ref{general} clearly implies $$
\EE(X)=\sum_{m\le
n}\Lambda(m)\EE(I_{\A}(m)).$$ Since $
\EE(I_{\A}(m))=\PP(\A\cap \{m,2m,\dots\}\ne \emptyset)=1-\prod_{r\le
n/m}\PP(rm\not \in \A)=1-(1-\delta)^{\lfloor n/m\rfloor}
$, we obtain
\begin{equation}\label{E}
\EE(X)=\sum_{m\le n}\Lambda(m)\left (1-(1-\delta)^{\lfloor
n/m\rfloor}\right ).
\end{equation}
 We observe that $\lfloor n/m
\rfloor=r$ whenever $\frac{n}{r+1}<m \leq\frac nr $, so we split the
sum into intervals $J_r=(\frac{n}{r+1},\frac nr ]$, obtaining
\begin{align*}
  \EE\left(X\right) &=\sum_{r\geq 1} (1-(1-\delta)^r) \sum_
{m\in J_r} \Lambda(m)\\
&= \sum_{r\geq 1} (1-(1-\delta)^r) \bigg(\psi\Big(\frac
nr\Big)-\psi\Big(\frac {n}{r+1}\Big)\bigg)\\
&= \delta \sum_{r\geq 1}\psi\Big(\frac nr\Big) (1-\delta)^{r-1}\\
&= \delta n \sum_{r\geq 1} \frac {(1-\delta)^{r-1}}{r} + \delta
\sum_{r\geq 1}
R\left(\frac{n}{r}\right)(1-\delta)^{r-1}.
\\
&= n\frac
{\delta\log(\delta^{-1})}{1-\delta}+ \delta
\sum_{r\geq 1}
R\left(\frac{n}{r}\right)(1-\delta)^{r-1}.
\end{align*}
\end{proof}
\begin{cor}\label{coro} If $\delta=\delta(n)<1$ and $\delta n\to \infty$ then
$$\EE\left (X\right
)= n\frac{\delta\log(\delta^{-1})}{1-\delta}\left (1+O\left (e^{-C\sqrt{\log (\delta n)}}\right )\right ).
$$
for some constant $C>0$.
\end{cor}
\begin{proof}
We estimate the absolute value of sum appearing in Proposition \ref{expectation}. For any positive integer $T$ and using that $|R(y)|<2y$ for all $y>0$  we have
 \begin{align*}\sum_{r\ge 1}|R\left(n/r\right)|
(1-\delta)^{r-1} =& \sum_{1\le r\le
T}|R\left(n/r\right)|(1-\delta)^{r-1}
+\sum_{r\ge T+1
}|R\left(n/r\right)|(1-\delta)^{r-1}\\
\le & n\sum_{1\le r\le
T}\frac{|R\left(n/r\right)|}{(n/r)}\frac{(1-\delta)^{r-1}}r
+2n\sum_{r\ge T+1
}\frac{(1-\delta)^{r-1}}r\\
\le &\ n\left (\max_{x\ge n/T}\frac{|R(x)|}{x}\right )\sum_{1\le r\le T}\frac{(1-\delta)^{r-1}}r+2n\sum_{r\ge T+1}\frac{(1-\delta)^{r-1}}r\\
\le &\ n \frac{\log (\delta^{-1})}{(1-\delta)}\left (\max_{x\ge n/T}\frac{|R(x)|}{x}\right )+\frac{2n}{T+1}\frac{(1-\delta)^{T}}{\delta}
\end{align*}
Taking into account that $(1-\delta)^{T}<e^{-\delta T}$ and the known estimate $$\max_{x>y}\frac{|R(x)|}x
\ll e^{-C_1\sqrt{\log y}}$$ for the error term in the PNT, we have
 \begin{align*}\sum_{r\ge 1}|R\left(n/r\right)|
(1-\delta)^{r-1} \ll & \ n \frac{\log (\delta^{-1})}{(1-\delta)}e^{-C_1\sqrt{\log(n/T)}}+n\frac{e^{-\delta T}}{\delta T}.
\end{align*}
Thus we have proved that for any positive integer $T$ we have
$$\mathbb E(X)=n\frac{\delta\log(\delta^{-1})}{1-\delta}\left (1+O\left ( e^{-C_1\sqrt{\log(n/T)}}\right )+O\left (\frac{1-\delta}{\log(\delta^{-1})}\frac{e^{-\delta T}}{\delta T}   \right )\right ).$$
We take
$T\asymp \delta^{-1}\sqrt{\log(\delta n)}$ to minimize the error term. To estimate the first error term we observe that $\log(n/T)\gg \log (\delta n/\sqrt{\log(\delta n)})\gg \log (\delta n)$, so $e^{-C_1\sqrt{\log(n/T)}}\ll  e^{-C\sqrt{\log(\delta n)}}$ for some constant $C$. To bound the second error term we simply  observe that $\delta T>1$ and  that $\frac{1-\delta}{\log(\delta^{-1})}\le 1$ and we get a similar upper bound.
\end{proof}
\subsection{Variance}\label{subsec:var}
\begin{prop}\label{var}For the random variable $X=\psi(\A)$ in $B(n,\delta)$ we have
 $$V(X)\ll \delta n \log^2
n.$$
\end{prop}
\begin{proof}By linearity of expectation we have that
\begin{align*}
 V(X)&=\EE\left(X^2\right
)-\EE^2\left(X\right)\\
&=\sum_{m,l\leq n}\Lambda(m)\Lambda(l)\left
(\EE\left(I_{\A}(m)I_{\A}(l)\right)-\EE\left(I_{\A}
(m)\right)\EE\left(I_{\A}(l)\right)\right
).
\end{align*}
We observe that if $\Lambda(m)\Lambda(l)\ne 0$ then $l\mid m,\ m\mid
l$ or $(m,l)=1$. Let us now study the term $\EE(I_{\A}(m)I_{\A}(l))$
in these cases.
\begin{enumerate}[(i)]
 \item If $l\mid m$ then
$$\EE(I_{\A}(m)I_{\A}(l))=1-(1-\delta)^{\lfloor n/m\rfloor}.$$
\item If $(l,m)=1$ then
$$\EE(I_{\A}(m)I_{\A}(l))=1-(1-\delta)^{\lfloor
n/m\rfloor}-(1-\delta)^{\lfloor n/l\rfloor}+
(1-\delta)^{\lfloor n/m\rfloor+\lfloor n/l\rfloor-\lfloor
n/ml\rfloor}.$$
\end{enumerate}
Both of these relations are subsumed in
\begin{equation*}
\EE(I_{\A}(m)I_{\A}(l))=1-(1-\delta)^{\lfloor
n/m\rfloor}-(1-\delta)^{\lfloor n/l\rfloor}+ (1-\delta)^{\lfloor
n/m\rfloor+\lfloor n/l\rfloor-\lfloor n(m,l)/ml\rfloor}.
\end{equation*}
Therefore, it follows from \eqref{E} that for each term in the sum we have
\begin{eqnarray*}\Lambda(m)\Lambda(l)\left(\EE\left(I_{\A}(m)I_{\A}
(l)\right)-\EE\left(I_{\A}(m)\right)\EE\left(I_{\A}(l)\right)\right
)\\
= \Lambda(m)\Lambda(l)(1-\delta)^{\lfloor n/m\rfloor+\lfloor
n/l\rfloor-\lfloor n(m,l)/ml\rfloor} \left (1-(1-\delta)^{\lfloor
n(m,l)/ml\rfloor  } \right ).\end{eqnarray*}
Finally, by using the inequality
$1-(1-x)^r\le rx$ we have
\begin{equation*}
\Lambda(m)\Lambda(l)\left(\EE\left(I_{\A}(m)I_{\A}
(l)\right)-\EE\left(I_{\A}(m)\right)\EE\left(I_{\A}(l)\right)\right
)\le \delta n\frac{\Lambda(l)}{l}\frac{\Lambda(m)}m(m,l),
\end{equation*}
and therefore:
$$
 V(X)\le  2\delta n\sum_{1\le l\le m\le
 n}\frac{\Lambda(l)}{l}\frac{\Lambda(m)}m(m,l).
$$
We now split the sum according to $l\mid m$ or $(l,m)=1$ and estimate each one separately.
\begin{eqnarray*}\sum_{\substack{1\le l\le m\le
 n\\ l\mid m}}\frac{\Lambda(l)}{l}\frac{\Lambda(m)}m(m,l)=\sum_{p\le n}\sum_{1\le j\le i}\frac{\log
p}{p^i}\frac{\log p}{p^i}p^j\le \sum_{p\le n}\sum_{1\le i}\frac{i\log^2
p}{p^i}
\ll \log^2n,
 \end{eqnarray*}
\begin{eqnarray*}\sum_{\substack{1\le l\le m\le
 n\\ (l,m)=1}}\frac{\Lambda(l)}{l}\frac{\Lambda(m)}m(m,l)\le \left (\sum_{1\le l\le n}\frac{\Lambda(l)}l\right )\left (\sum_{1\le m\le n}\frac{\Lambda(m)}m\right )
\ll \log^2n,
 \end{eqnarray*}
as we wanted to prove.\end{proof}
We finish the proof of Theorem \ref{TH1} by observing that  $V(X)=o(\EE(X)^2)$ when $\delta n\to \infty$, so $X\sim \EE(X)$ asymptotically almost surely.
\section{Chebyshev's function for random sets in $S(n,k)$. Proof of Theorem \ref{TH2}}
Let us consider again the random variable $X=\psi(\A)$, but in the model $S(n,k)$. From now on $\EE_k(X)$ and $V_k(X)$ will denote the expected value and the variance of $X$ in this probability space. Clearly, for $s=1,2$ we have
\begin{eqnarray*}
\EE_k(X^s)&=&\frac 1{\binom
nk}\sum_{|\A|=k }\psi^s(\A)\\ V_k(X)&=&\frac 1{\binom nk}\sum_{|\A|=k}\left
(\psi(\A)-\EE_k(X)\right
)^2\end{eqnarray*}
\begin{lemma}\label{k-j}
For $s=1, 2$  and $1\le j<k$ we have that
$$\EE_j(X^s)\le \EE_k(X^s)\le
\EE_j(X^s)+(k^s-j^s)\log^s n.$$
\end{lemma}

\begin{proof} In order to prove the lower bound it is enough to consider the case $j = k-1$. Observe that the function $\psi$ is monotone with respect to inclusion, i.e. $\psi\left(\A\cup \{a\}\right)\geq \psi(\A)$ for any $\A,\{a\}\subseteq [n]$.
  Using this we get
  $$
  \sum_{|A| = k-1}\psi^s(\A) \leq \frac{1}{n-k+1}\sum_{a\in [n]\setminus
  A}\psi^s(\A\cup\{a\}) = \frac{k}{(n-k+1)}\sum_{|A'| = k}\psi^s(\A').
  $$
  Inequality then follows from $\binom n {k-1}= \frac{k}{(n-k+1)} \binom n k$.

For the second inequality  we observe that for any set $\A\in \bnk$
and any partition into two sets $\A=\A'\cup \A''$ with $|\A'|=j,\
|\A''|=k-j$ we have that $\psi(\A)\le \psi(\A')+\psi(\A'')\le
\psi(\A')+(k-j)\log n$. Similarly, \begin{eqnarray*}\psi^2(\A)&\le &
(\psi(\A')+(k-j)\log n)^2\\ &= &\psi^2(\A')+2\psi(\A')(k-j)\log
n+(k-j)^2\log^2n\\ &\le &\psi^2(\A')+2j(k-j)\log^2 n+(k-j)^2\log^2n\\
&= &\psi^2(\A')+(k^2-j^2)\log^2n.\end{eqnarray*} Thus, for $s=1,2$
we have
\begin{eqnarray*} \psi^s(\A)&\le &\binom kj^{-1} \sum_{\substack{\A'\subset \A\\
|\A'|=j}}\left
(\psi^s(\A')+(k^s-j^s)\log^s n\right )\\&\le &\binom kj^{-1}
\Big(\sum_{\substack{\A'\subset \A\\ |\A'|=j}}\psi^s(\A')\Big)+(k^s-j^s)\log^s
n.\end{eqnarray*} Then,
\begin{eqnarray*}\sum_{|\A|=k}\psi^s(\A)&\le &\binom
kj^{-1}\sum_{|\A|=k}\sum_{\substack{\A'\subset \A\\
|\A'|=j}}\psi^s(\A')+\binom
nk(k^s-j^s)\log^s n\\
&= &\binom kj^{-1}\sum_{|\A'|=j}\psi^s(\A')\sum_{\substack{\A'\subset
\A\\
|\A|=k}}1+\binom nk(k^s-j^s)\log^s n\\
&= &\binom kj^{-1}\binom{n-j}{k-j}\sum_{|\A'|=j}\psi^s(\A')+\binom
nk(k^s-j^s)\log^s n\\&= &\frac{\binom nk}{\binom
nj}\sum_{|\A'|=j}\psi^s(\A')+\binom nk(k^s-j^s)\log^s n,
\end{eqnarray*}
and the second inequality holds.
\end{proof}
\begin{prop}\label{Tcomb} For $s=1,2$ we have that
$$\EE_k(X^s)=\EE(X^s)+O(k^{s-1/2}\log^s n)
$$
where $\EE(X^s)$ denotes the expectation of $X^s$  in $
B(n,k/n)$ and $\EE_k(X^s)$ the expectation in $S(n,k)$.
\end{prop}
\begin{proof}
Observe that for $s=1,2$ we have
\begin{eqnarray*}\EE(X^s)-\EE_k(X^s)&=&-
\EE_k(X^s)+\sum_{j=0}^n\left(\frac{k}{n}
\right)^j\left(1-\frac{k}{
n}\right)^{n-j}
\sum_{|\A|=j}\psi^s(\A)\\&=&-\EE_k(X^s)+\sum_{j=0}^n
\left(\frac{k}{n}\right)^j\left(1-\frac{k}{n}\right)^{n-j}\binom
nj\EE_j(X^s)\\
&=&\sum_{j=0}^n\left(\frac{k}{n}\right)^j\left(1-\frac{k}{n}\right)^{
n-j}\binom nj\left
(\EE_j(X^s)-\EE_k(X^s)\right),
\end{eqnarray*}
for $s=1,2$. Using Lemma \ref{k-j} we get
\begin{equation}\label{Ejk}|\EE_k(X^s)-\EE(X^s)|\le \log^sn\sum_{j=0}^n\left(\frac{k}{n}\right)^j\left(1-\frac{k}{n}\right)^{
n-j}\binom nj|j^s-k^s|.\end{equation}
The sum in~\eqref{Ejk} for $s=1$ is $\mathbb{E}(|Y-\mathbb{E}(Y)|)$, where $Y\sim \text{Bin}(n,k/n)$ is the binomial distribution of parameters $n$ and $k/n$. Chauchy-Schwarz inequality for the expectation implies that this quantity is bounded by the standard deviation of the binomial distribution.
 \begin{equation}\label{ss}\sum_{j=0}^n\left(\frac{k}{n}\right)^j\left(1-\frac{k}{n}\right)^{
n-j}\binom nj|j-k|\le   \sqrt{n(k/n)(1-k/n)}\le \sqrt k,\end{equation}
which proves Proposition \ref{Tcomb} for $s=1$.

To estimate the sum in~\eqref{Ejk} for $s=2$, we split the expression in two terms: the sum indexed by $j\leq 2k$ and the one with $j>2k$.
We use  \eqref{ss} to get
\begin{eqnarray*}
\sum_{j\le 2k} \left(\frac{k}{n}\right)^j\left(1-\frac{k}{n}\right)^{
n-j}\binom nj|j^2-k^2|&\le &3k\sum_{j=0}^n\left(\frac{k}{n}\right)^j\left(1-\frac{k}{n}\right)^{
n-j}\binom nj|j-k|\\ &\le & 3k^{3/2}.
\end{eqnarray*}
On the other hand,
\begin{eqnarray*}\sum_{j> 2k} & &\left(\frac{k}{n}\right)^j\left(1-\frac{k}{n}\right)^{
n-j}\binom nj|j^2-k^2|\\
\le& &\sum_{l\ge 2}(l+1)^2k^2\sum_{lk<j\le (l+1)k} \left(\frac{k}{n}\right)^j\left(1-\frac{k}{n}\right)^{
n-j}\binom nj\\
\le & &\sum_{l\ge 2}(l+1)^2k^2\ \mathbb P(Y>lk)
\end{eqnarray*}
where, once again,  $Y\sim \text{Bin}(n,k/n)$. Chernoff's Theorem implies that for any $\epsilon>0$ we have
$$\mathbb P(Y>(1+\epsilon)k)\le e^{-\epsilon^2k/3}.$$ Applying this inequality to $\mathbb P(Y>lk)$ we get
\begin{eqnarray*}\sum_{j> 2k} & &\left(\frac{k}{n}\right)^j\left(1-\frac{k}{n}\right)^{
n-j}\binom nj|j^2-k^2|\\
\le & &\sum_{l\ge 2}(l+1)^2k^2 e^{-(l-1)^2k/3}\ll k^2e^{-k/3}\ll k^{3/2}.
\end{eqnarray*}
\end{proof}
The next corollary proves the first part of Theorem \ref{TH2}.
\begin{cor}\label{coro2} If $k=k(n)<n$ and $k\to \infty$ then
$$\EE_k(X)=k\frac{\log(n/k)}{1-k/n}\left (1+O\left( e^{-C\sqrt{\log k}}\right) \right )$$
\end{cor}
\begin{proof}
Proposition \ref{Tcomb} for $s=1$ and Corollary \ref{coro} imply that $$\EE_k(X)=k\frac{\log(n/k)}{1-k/n}\left (1+O\left(e^{-C\sqrt{\log k}}\right)+O\left( k^{-1/2}\right) \right )$$ and clearly $k^{-1/2}=o\left(e^{-C\sqrt{\log k}}\right)$ when $k\to \infty$.
\end{proof}
To conclude the proof of Theorem~\ref{TH2} we combine Proposition \ref{var} and Proposition \ref{Tcomb} to estimate the variance $V_k(X)$ in $S(n,k)$:
\begin{align*}
 V_k(X)&=\EE_k(X^2)-\EE_k^2(X)\\&=V(X)+\left
(\EE_k(X^2)-\EE(X^2)\right )+\left
(\EE(X)-\EE_k(X)\right )\left
(\EE(X)+\EE_k(X)\right )
\\ &\ll k\log^2n+\left (k^{1/2}\log n\right )\left ( k\log
n\right )
\\ & \ll  k^{3/2}\log^{2}n.
\end{align*}
The second assertion of Theorem \ref{TH2} is a consequence of the estimate  $V_k(X)=o\left (\EE_k^2(X)\right )$ when $k\to \infty$.
\subsection{The case when $k$ is constant}
The case when $k$ is constant and $n\to\infty$ is not relevant for our original motivation but we give a brief analysis for the sake of completeness.  In this case Fern\'{a}ndez and Fern\'{a}ndez \cite{FF} have proved that $\mathbb E_k(\psi(A))=k\log n+C_k+o(1)$ where $C_k=-k+\sum_{j=2}^k\binom kj(-1)^j\frac{\zeta'(j)}{\zeta(j)}$. Actually, they consider the probabilistic model with $k$ independent choices in $\{1,\dots, n\}$, but when $k$ is fixed it does not make big differences because the probability of a repetition  between the $k$ choices is tiny.

It is easy to prove that with probability $1-o(1)$ we have that $\psi(A)\sim k\log n$. To see this we observe that
\[a_1\cdots a_k \prod_{i<j}(a_i,a_j)^{-1}\le \lcm(a_1,\dots, a_k)\le a_1\cdots a_k\le n^k,\] so $\sum_{i=1}^k\log a_i-\sum_{i<j}\log(a_i,a_j)\le \psi(A)\le k\log n$.

Now, let us note that $\mathbb P(a_i\le n/\log n \text{ for some } i=1,\dots,k)\le k/\log n$
and that $\mathbb P((a_i,a_j)\ge \log n)\le \sum_{d>\log n}\mathbb P(d\mid a_i,\ d\mid a_j)\le \sum_{d>\log n}\frac 1{d^2}<\frac 1{\log n}.$
These observations imply that with probability at least  $1-\frac{k+\binom k2}{\log n}$ we have that
$$ k\log n\left (1-O\left (\log \log n/\log n\right )\right )\le \psi(A)\le k\log n.$$
The analysis in the model $B(n,\delta)$ when $\delta n\to  c$ can be done using again Proposition \ref{expectation}.
\begin{eqnarray*}
\EE\left(\psi(A)\right)=n\frac
{\delta\log(\delta^{-1})}{1-\delta}+\delta\sum_{r< n/\log n}
R\left(\frac{n}{r}\right) (1-\delta)^{r-1}+\delta\sum_{n/\log n \le r\le n }
R\left(\frac{n}{r}\right) (1-\delta)^{r-1}
\end{eqnarray*}
We use the estimate $R(x)\ll x/\log x$ in the first sum and the estimate $R(x)\ll x$ in the second one. We have
\begin{eqnarray*}
\EE\left(\psi(A)\right)&=&c\log n+O(1)+O\left (\frac{c}{\log \log n}\sum_{r< \frac n{\log n}}
\frac{(1-\delta)^{r-1}}r\right)+O\left (c\sum_{\frac n{\log n} \le r\le n }
 \frac{(1-\delta)^{r-1}}r \right )\\
 &=&c\log n+O\left (\frac{c\log \delta}{\log \log n}   \right )+O\left (c\log \log n\right )\\
 &=&c\log n(1+o(1)).
\end{eqnarray*}
Of course in this model we cannot expect concentration around the expectation because for example the probability that $A$ is the empty set tends to a positive constant, $\mathbb P(A=\emptyset)\to e^{-c}$, and then $\mathbb P(\psi(A)=0)\to e^{-c}$.

\section{The least common multiple of the values of a polynomial}

Chebyshev's
function could be also generalized to $$\psi_q(n)=\log \lcm \left \{q(k):\ 1\le k,\
1\le q(k)\le n \right\}$$ for a given polynomial $q(x)\in
\mathbb{Z}[x]$ and it is natural to try to obtain the asymptotic
behavior for $\psi_q(n)$. Some progress has been made in this
direction. While the Prime Number Theorem is equivalent to the asymptotic $\psi_q(n)\sim n$ for $q(x)=x$, Paul Bateman noticed that the Prime Number Theorem for arithmetic
progressions could be exploited to obtain the asymptotic estimate when
$q(x)=a_1x+a_0$ is a linear polynomial and proposed it as a problem \cite{B} in the American Mathematical Monthly:
\begin{equation*}\label{linear}\psi_q(n)\sim \frac{n}{a_1}\frac
m{\phi(m)}\sum_{\substack{1\le l\le m\\(l,m)=1}}\frac 1l
,\end{equation*}
where $m=a_1/(a_1,a_0)$. The first author~\cite{C} has extended this
result to quadratic polynomials. For a given irreducible quadratic
polynomial $q(x)=a_2x^2+a_1x+a_0$ with $a_2>0$ the
following asymptotic estimate holds:
\begin{equation}\label{cille1}
\psi_q(n) = \frac 12\left(n/a_2\right)^{1/2}\log
\left(n/a_2\right)
+B_q\left(n/a_2\right)^{1/2}+o(n^{1/2}),
\end{equation}
where the constant $B_q$ depends only on $q$. In the particular
case of $q(x)=x^2+1$, he got $\psi_q(n)=\frac 12n^{1/2}\log
n+B_qn^{1/2}+o(n^{1/2})$ with
$$ B_q=\gamma-1-\frac{\log
2}{2}-\sum_{p\neq 2}\frac{(-1)^{\frac{p-1}2}\log p}{p-1}, $$
where $\gamma$ is the Euler constant and the sum is considered over all odd prime
numbers. It has  been proved \cite{RSZ} that the error
term in \eqref{cille1} for $q(x)=x^2+1$ is $O\left(n^{1/2}\left(\log n\right)^{-4/9+\epsilon}\right)$
for each $\epsilon>0$. When $q(x)$ is a reducible polynomial the
behavior is, however, different. In this case it is known (see Theorem 3
in~\cite{C}) that:
\begin{equation*}\label{red}\psi_q(n)\sim c
n^{1/2}\end{equation*} where $c$ is an explicit constant depending
only on $q$. For example for $q(x)=x^2-1$ the constant is $c=1$.

The asymptotic behavior of $\psi_q(n)$ remains unknown for
irreducible polynomials of higher degree.
\begin{conj}[Cilleruelo \cite{C}]\label{co} Let $q(x)$ be an irreducible polynomial of degree $d\ge 3$. Then
\begin{equation}\psi_q(n)\sim (1-1/d)\left(n/a_d\right)^{1/d}\log
\left(n/a_d\right),\end{equation}
where $a_d>0$ is the coefficient of $x^d$ in $q(x)$.\end{conj} For example, this
conjecture would imply $\psi_q(n)\sim  \frac 23 n^{1/3}\log n$ for
$q(x)=x^3+2$.

We observe that $\psi_q(n)=\psi(A_{q,n})$ where $A_{q,n}:=\{q(k):\ 1\le k,\ 1\le q(k)\le n\}$. It is natural to wonder whether for a given polynomial $q(x)$  the asymptotic
 $\EE_k(X)\sim \psi_q(n)$ holds, when $n\to
\infty$, where $k=|A_{q,n}|$ and $X=\psi(A)$ for a random set $A$ of $k$ elements in $\{1,\dots,n\}$.

However, consider for example the polynomials $q(x)=x^2-1$ and $q(x)=x^2+1$. In both cases $|A_{q,n}|\sim \sqrt n$ but the asymptotic behaviors of $\psi_q(n)$ are distinct: \begin{equation*}\psi_q(n)\sim \begin{cases}\quad \sqrt n \qquad \quad \text{ when } q(x)=x^2-1\\ \frac 12\sqrt n\log n \quad\text{ when } q(x)=x^2+1.\end{cases}\end{equation*} So, what is the typical behavior of $\psi(A)$ when $|A|\sim \sqrt n$? Is it like in the reducible case or like in the irreducible one? Maybe neither of them represent the typical behavior of a random set.

This question was the original motivation of this work. Theorem \ref{TH2} with $k=|A_{q,n}|=\sqrt{n/a_2}+O(1)$ gives $$\mathbb E_k(X)=k\frac{\log(n/k)}{1-k/n}\left (1+O\left( e^{-C\sqrt{\log k}}\right) \right )=\frac 12(n/a_2)^{1/2}\log (n/a_2)+o\left (n^{1/2}\right ).  $$
This shows that, when $q(x)$ is an irreducible quadratic polynomial, the asymptotic behavior of $\psi_q(n)$ coincides with $\psi(A)$, for almost all sets of  size $|A_{q,n}|$. Theorem \ref{TH2} also supports Conjecture \ref{co} for any $q(x)=a_dx^d+\cdots +a_0$  irreducible polynomial of degree $d\ge 3$.

Nevertheless, there
are some differences in the second term. For example,  if
$q(x)=x^2+1$, we have
$$
\psi_q(n)=\frac 12n^{1/2}\log n+B_qn^{1/2}+o(n^{1/2}),$$
for $B=-0.06627563..$. On the other hand,  Theorem \ref{TH2} implies that in corresponding model $S(n,k)$ with $k=|\A_{q,n}|=\lfloor \sqrt{n-1}\rfloor$ we have that
$$\psi(A)=\frac 12n^{1/2}\log n+o(n^{1/2})$$
almost surely. In other words, when $q(x)$ is an irreducible quadratic polynomial, the asymptotic behavior of $\psi_q(n)$ is the same that $\psi(A)$ in the corresponding model $S(n,k)$. But, the second term is not typical unless $B_q=0$. Probably $B_q\ne 0$ for any irreducible quadratic polynomial $q(x)$ but we have not found a proof.

\subsubsection*{Acknowledgments:} This work was supported by grants MTM
2011-22851 of MICINN, ICMAT Severo Ochoa project SEV-2011-0087 and Research
Council of Lithuania Grant No. MIP-068/2013/LSS-110000-740.
Part of this work was made during third author's visit at Universidad
Aut\'{o}noma de Madrid, he would like to thank the people of the Mathematics Department and
especially Javier Cilleruelo for their warm hospitality. The authors
are also grateful to Surya Ramana and Javier Cárcamo for useful
comments, and especially to Ben Green whose suggestions improved the clarity and exposition of the results.

\end{document}